\documentclass[a4paper, twoside,12pt]{article}
\usepackage{fancyhdr}
\usepackage{fnpos}
 \usepackage[english]{babel}        
\usepackage[T1]{fontenc}          
\usepackage{graphicx}             
\usepackage{makeidx}
\usepackage{fancybox}
\usepackage{framed}
\usepackage{fancyhdr}
 \usepackage{pstricks,pst-plot,pstricks-add}
\usepackage[margin=1in]{geometry}
\usepackage{graphicx}
\usepackage{titlesec}
\usepackage{amsmath}
\usepackage{amsfonts}   
\usepackage{amssymb}    
\usepackage{amsthm}
\usepackage{dsfont}
\usepackage{mathtools}
\usepackage{verbatim}   
\usepackage{color}
\usepackage{listings}
\usepackage{graphicx}
\usepackage{amsmath}
\usepackage{amsmath,amssymb,color,inputenc,euscript,graphicx,psfrag}

\newtheorem{thm}{Theorem}[section]

\newtheorem{lem}[thm]{Lemma}
\newtheorem{prop}[thm]{Proposition}

\numberwithin{equation}{section}

\renewcommand{\thefootnote}

 {  }{  }

\author { B\'echir Amri   }

\title{$L^p$-Convergence of  Fourier-Heckman-Opdam Expansions }
\date{ }

\begin{document}
 \maketitle
 \begin{center}
 Department of Mathematics, Faculty of Sciences. Taibah university 
Al-madinah Al-Munawarah. Arabia Saudi.\\
E-mail: bamria@taibahu.edu.sa, 
\end{center}
 \begin{abstract}
 
 We study the 
$L^p$-convergence of Fourier expansions in terms of non-symmetric Heckman–Opdam polynomials of type $A_1$. Using kernel estimates and duality arguments, we prove that the partial sums converge in $ L^p([-\pi,\pi],dm_k)$ 
for 
$$2-\frac{1}{k+1} < p < 2+\frac{1}{k}.$$
This extends classical $L^p$-convergence results for trigonometric Fourier series to the Heckman–Opdam setting.
\footnote{Key words: Heckman-Opdam polynomials, Orthogonal polynomials, Integral operators.  \\
2020 Mathematics Subject Classification: 33C52, 33C45, 42B10.}
 \end{abstract}
 \section{Introduction}
 Fourier analysis has long been a central tool in harmonic analysis and mathematical physics. In particular, the convergence of Fourier series in $L^p$ -spaces is a classical topic with deep connections to orthogonal polynomials and special functions \cite{MS, PP, P}. In recent decades, generalizations of classical Fourier analysis to settings associated with root systems and reflection groups have attracted considerable interest.
Heckman–Opdam polynomials provide a natural generalization of classical Jacobi polynomials and appear as eigenfunctions of Cherednik-type differential–difference operators \cite{O1, O2}. The non-symmetric Heckman–Opdam polynomials of type 
$A_1$  form an orthogonal basis of  $L^2([-\pi,\pi],dm_k)$,
$dm_k(x)=|\sin x|^{2k}dx$, and exhibit rich algebraic and analytic structures, including close relations with non-symmetric Jack polynomials.
In this work, we investigate the $L^p$-convergence of Fourier expansions with respect to the non-symmetric Heckman–Opdam polynomials. Specifically, for 
 $f\in  L^p([-\pi,\pi],dm_k)$ 
, we consider its Fourier–Heckman–Opdam expansion and study the convergence of its partial sums.  Following the approach developed in \cite{P}, and using integral kernel representations together with suitable estimates, we show that the partial sums converge to  $f$ for 
$$2-\frac{1}{k+1} < p < 2+\frac{1}{k}.$$
 
 \section{Non-symmmetric Heckmann-Opdam polynomials  of type $A_1$ }
  We begin by reviewing and presenting several results concerning the non-symmetric Heckman–Opdam polynomials of type $A_1$. For the general theory of Heckman–Opdam polynomials associated with arbitrary root systems, we refer to the standard literature \cite{O1,O2}.
\par Let $k\geq 0$, the Cherednick operator of type $A_1$ is given by
\begin{eqnarray}\label{T}
T^{k}(f)(x)=\dfrac{df}{dx}(x)+ 2k\;\frac{ f(x)-f(-x)}{1-e^{-2x}}-   k f(x), \qquad f\in C^1(\mathbb{R}).
 \end{eqnarray}
 We define the weight function
$$dm_k(x)=|\sin x |^{2k} dx  $$
and the inner product on $L^2([-\pi,\pi],dm_k(x))$
$$(f,g)_k=\int_{-\pi}^{\pi}f(x)\overline{g(x)}dm_k(x),$$
with its associate norm $\|.\|_{2,k}$. Introduce a partial ordering $\triangleleft$ on $\mathbb{Z}$ as follows:
\begin{equation*}
  j\triangleleft n  \Leftrightarrow \left\{
    \begin{array}{ll}
    |j| < |n|\;\text{and} \;|n|-|j|\in 2\mathbb{Z}^+   , & \hbox{} \\
    or   , & \hbox{} \\
      |j| = |n|\;\text{and} \; n <j . & \hbox{.}
    \end{array}
  \right.
\end{equation*}
The non-symmetric  Heckman-Opdam polynomials $E_n^k$, $n\in \mathbb{Z}$ are characterized by the following conditions:
\begin{eqnarray}\label{def}
 (a)&&    E_n^{k}(z)=\textrm{e}^{nz}+ \sum_{j\triangleleft n }c_{n,j}\; \textrm{e}^{jz},\qquad z\in \mathbb{C}\qquad\qquad\qquad\qquad\qquad\qquad\qquad\qquad
 \\ (b) && (E_n^{k}(\textrm{i}x),\textrm{e}^{\textrm{i}jx})_{k}=0,\quad \text{for any } \; j\triangleleft n.
\end{eqnarray}
  It follows from Corollary 5.2 and
Proposition 6.1 in \cite{Sah} that the coefficients  $c_{n,j}$  are all real and non-negative.
The  polynomials $E_n^k$ diagonalize simultaneously the  Cherednik operators. In particular, they satisfy 
\begin{equation}\label{TE}
  T^{k} ( E_n^{k})= n_k \; E_n^{k}, \qquad n\in \mathbb{Z},
\end{equation}
 where $ n_k=n+k $ if $n >0$ and  $n_k=n-k$ if $n\leq 0$.
We observe that,
\begin{equation}\label{13}
 E_0^k(x)=1,\quad \text{and}\quad E_1^k(x)=e^x.
\end{equation}
 Moreover, the family of trigonometric polynomials
$\{ E_n^{k}(ix), \;n\in \mathbb{Z}\}$    is an orthogonal basis of 
  $L^2([-\pi,\pi],dm_k(x))$.\\
 \par The Heckman-Opdam-Jacobi  polynomials  $P_n^k$, $n\in \mathbb{Z}_+$, are defined by
  $$P_n^{k}(x) =\frac{1}{2} \Big(E_n^{k}(x)+E_n^{k}(-x)\Big),$$
and they satisfy 
$$ (T^k)^2P_n^k=(n+k)^2 P_n^k.$$
In particular,  $P_n^k$ is an  eigenfunction of the operator  $L_k$, which coincides with the even part of $(T^k)^2$, given by 
$$L_k(f)=f''+2k\; \frac{1+e^{-2x}}{1-e^{-2x}}\;f'(x)+k^2f(x)$$
The Heckman-Opdam-Jacobi  polynomials are expressed via the hypergeometric function $_2F_1$
as
\begin{equation}\label{P}
P_n^k(x)=P_n^k(0)\;_2F_1(n+2k,-n,k+1/2,-\sinh^2(x/2))
\end{equation}
where, for all $n\in \mathbb{Z}^+$,
$$ P_n^k(0)=E_n^{k}(0)= \frac{\Gamma(k+1)\Gamma(n+2k)}{\Gamma(2k+1)\Gamma(n+k)}=\frac{\Gamma(k)\Gamma(n+2k)}{2\Gamma(2k)\Gamma(n+k)}.$$

They can also be written in terms of the ultraspherical (Gegenbauer) polynomials $C_n^k$ as
\begin{equation}\label{ultra}
P_n^k(ix)=\frac{P_n^k(0)}{C_n^k(1)} C_n^k(\cos x)= \frac{\Gamma(k+1)}{\Gamma(2k+1)}\frac{n!}{\Gamma(n+k)}C_n^k(\cos x).
\end{equation}
Here the polynomials $C_n^k$ $C_n^k$are defined via their generating function,
$$
\sum_{n=0}^{\infty} C_n^{k}(x)\, t^n
=
\frac{1}{(1 - 2 x t + t^2)^{\lambda}}, 
\qquad |t|<1.
$$
(See \cite{Sz}, for further properties of ultraspherical polynomials.)

  \par An important identity that can follows from  (\ref{P}) is
\begin{equation}\label{pn}
\frac{dP_n^k}{dx}(x)= 2 nP_{n-1}^{k+1}(x)\sinh x, \quad n\geq 1.
\end{equation} 
Further,  making use of the identities:
$$\frac{dE_n^k}{dx} (x)+2k \frac{E_n^k(x)-E_n^k(-x)}{1-e^{-2ix}}=(n+2k)E_n^k(x)$$
$$\frac{dE_n^k}{dx}-2k \frac{E_n^k(x)-E_n^k(-x)}{1-e^{2x}}=(n+2k)E_n^k(-x)$$
and subtracting the second  from the first we obtain  that 
$$\frac{dP_n^k}{dx}(x)=n\;\dfrac{E_n^k(x)-E_n^k(-x)}{2}.$$
 This yields, in view of (\ref{pn}), the representation formula:
\begin{equation}\label{jacobi}
   E_n^k(x)= P_{n}^k( x)+2\sinh x \;P_{n-1}^{k+1}(x),\quad n\geq 1.
\end{equation}
 \par It should be noted that the non-symmetric  Heckman-Opdam polynomials $E_n^k$ are closely  related to non-symmetric
Jack polynomials,  see \cite{Sah}.
If  $\lambda=(\lambda_1,\lambda_2)\in \mathbb{Z}\times \mathbb{Z} $ and $\mathcal{I}_{\lambda }^k$ is the correspondent non-symmetric
Jack polynomial then
$$E_{\lambda_2-\lambda_1}^k= \mathcal{I}_{\lambda}^k(e^{-x},e^{x}).$$
 \begin{prop}
 For all $n\in \mathbb{Z}$, we have
\begin{equation}\label{12}
   E_{n+1}^k(x)=e^xE_{-n}^k(-x).
\end{equation}
\end{prop}
\begin{proof}
In view of (\ref{13}), the identity (\ref{12}) holds for     for $n= -1,0$. Let $n\in \mathbb{Z}$, $n\neq -1,0$. Put
 $H_n^k(x)= e^xE_{-n}^k(-x)$. It is  enough to check that $ T^k(H_n)=(n+1)_kH_n$. We have
\begin{eqnarray*}
&&T^k(H_n)(x)
\\ &&\quad =e^xE_{-n}^k(-x)-e^x(E_{-n}^k)'(-x)+2k \; \frac{e^xE_{-n}^k(-x)-e^{-x}E_{-n}^k(x)}{1-e^{-2x}}-ke^xE_{-n}^k(-x)
\\&&\quad= -e^x \left((E_{-n}^k)'(-x)+2k\;\frac{E_{-n}^k(-x)-E_{-n}^k(x)}{1-e^{2x}} -k E_{-n}^k(-x) \right)+  e^xE_{-n}^k(-x)
\\&&\quad= ( 1- (-n)_k)e^xE_{-n}(-x)=(n+1)_ke^xE_{-n}(-x)=(n+1)_kH_n(x).
\end{eqnarray*}
Comparing the leading coefficients yields (\ref{12}). 
\end{proof}
\begin{prop}
 For $n \in \mathbb{Z}^+$, $n\neq 0$, we have
\begin{equation}\label{45}
 E_{-n}^k(x)=E_{n}^k(-x) +\frac{k}{n+k}E_n^k(x).
\end{equation}
  \end{prop}
\begin{proof}
Observe that  for $n\geq 1$ we have $T^k(E_n(-x))=-2kE_n(x)- (n+k)E_n(-x)$. Then
\begin{eqnarray*}
  T^k\left( E_n^k(-x)+\frac{k}{n+k}E_n^k(x)\right)&=&-2kE_{n}(x)- (n+k)E_n(-x)+ kE_{n}(x)
\\&=&-kE_{n}(x)-(n+k)E_n^k(-x)
\\&=&-(n+k)\left(E_n^k(-x)+\frac{k}{n+k}E_n^k(x)\right).
 \end{eqnarray*}
Equality  (\ref{45}) follows by comparing the highest coefficient.
\end{proof}
  \begin{prop}
  For all $n  \in \mathbb{Z}^+$, we have
  \begin{equation}\label{Ga}
   \|E_{n+1}^k\|_{2,k}=\|E_{-n}^k\|_{2,k}=  \pi2^{1-2k} n!\; \frac{\Gamma(n+2k+1)}{\Gamma(n+k+1)^2}. 
  \end{equation}
 Here   $\|E_{\nu}^k\|_{2,k}$ denotes the norm  of $E_{\nu}^k(ix)$ in the space $L^2([-\pi,\pi],dm_k)$.
  \end{prop}
\begin{proof}
In view of  (\ref{45}), we write
$$  E_n^k(-x)= E_{-n}^k(x)-\frac{k}{n+k}E_n^k(x), \quad n\geq 1\quad.$$
  Since $E_n^k$ and $E_{-n}^k$ are orthogonal, it follows that
$$\|E_n^k\|_{2,k}^2= \|E_{-n}^k\|_{2,k}^2 +\frac{k^2}{(n+k)^2}\|E_n^k\|_{2,k}^2$$
 and hence
 \begin{equation}\label{69}
\|E_{n+1}^k\|_{2,k}^2= \|E_{-n}^k\|_{2,k}^2 =
\frac{n(n+2k)}{(n+k)^2} \|E_n^k\|_{2,k}^2 .
 \end{equation}
Moreover, one easily checks that
 $$ \|E_1^k\|_{2,k}^2=\|E_{0}^k\|_{2,k}^2=   \pi2^{1-2k}\frac{\Gamma(2k+1)}{\Gamma(k+1)^2}=\dfrac{2\sqrt{\pi}\Gamma(k+\dfrac{1}{2})}{\Gamma(k+1)}.$$
The formula (\ref{Ga}) then follows.
\end{proof}
 \par In conclusion,  the explicit expressions of the Heckman–Opdam polynomials $E_n^k$ are given by:
$E_0^k = 1$, and for all $n\geq 1$,
\begin{equation}\label{cases}
\begin{cases}
\displaystyle
E_n^k(x)
= P_n^k(x) + 2\sinh(x)\, P_{n-1}^{\,k+1}(x), \\[2mm]
\displaystyle
E_{-n}^k(x)
= \frac{n+2k}{\,n+k\,}\, P_n^k(x)
- \frac{2n}{\,n+k\,}\, \sinh(x)\, P_{\,n-1}^{\,k+1}(x).
\end{cases}
\end{equation}

\par Let function $f\in L^p([-\pi,\pi],dm_k(x))$.The Fourier-Heckman-Opdam coefficients of $f$ are defined by
\begin{equation}\label{an}
a_{n}= \gamma_{n}^2\int_{-\pi}^{\pi}f(x) E_n^k(-ix)dm_k(x), \qquad n\in \mathbb{Z}.
\end{equation}
 where  $\gamma_{n}=\|E_n^k\|_{2,k}^{-1}$.
The Fourier–Heckman–Opdam series of $f$ is formally given by
$$f(x)\sim \sum_{n=-\infty}^{+\infty}a_n E_n^k(ix).$$
For  $N\in \mathbb{N}$, we define the $N$-th partial sum 
$$S_{N}(f)(x)=\sum_{n=-N}^{N}a_{n} E_n^k(ix).$$
Then  $S_N(f)$ can be represented as an integral operator:
\begin{equation}\label{SK}
 S_N(f)(x)=\int_{-\pi}^{\pi}K_N(x,y)f(y)dm_k(y).
\end{equation}
where 
 $$K_{N}(x,y)=\sum_{n=-N}^N \gamma_{n}^2 E_n^k(ix)E_n^k(-iy)$$
Throughout this work, for sake of simplicity in writing, we omit the dependence on $k$ in certain  notations whenever no ambiguity can arise. The norm on the space 
$L^p([-\pi,\pi],dm_k)$ will be denoted  by $\|.\|_{p,k}$. We shall denote by 
A a constant, which may vary from line to line.

 \section{ $L^p$-Convergence of Partial Sums}
The main purpose is to establish the $L^p$-convergence of the partial sums $\{S_N(f)\}_{N\in \mathbb{N}}$.
 \begin{thm}[Main result]\label{MR}
Let $f \in L^p([-\pi,\pi],dm_k)$ and assume that
$$
2-\frac{1}{k+1} < p < 2+\frac{1}{k}.$$
Then the sequence of partial sums $\{S_N(f)\}_{N\ge1}$ converges to $f$ in
$L^p([-\pi,\pi],dm_k)$, that is,
$$
\lim_{N\to\infty}\|S_N(f)-f\|_{p,k}=0.
$$
\end{thm}
 The proof is carried out by means of a sequence of auxiliary lemmas.
 \begin{lem}\label{prop1}
 For any integer $N\geq 0$
 and all $x,y\in\mathbb{R}$
 such that $x+y\notin 2\pi\mathbb{Z}$, 
 the following identity holds:
 \begin{equation}\label{sum}
 K_N(x,y)=\gamma_{N+1}^2\frac{e^{-i(x-y)} E_{N+1}^k(ix)E_{N+1}^k(-iy)-^2E_{N+1}^k(-ix)E_{N+1}^k(iy)}{1-e^{-i(x-y)}}.
 \end{equation}
 \end{lem}
 \begin{proof}  
  We first show that $K_N(x,y)$ is real valued.
By (\ref{45}),
\begin{eqnarray*}
&&\sum_{n=1}^N\gamma_{-n}^2 E_{-n}^k(ix)E_{-n}^k(-iy)\\
&&=\sum_{n=1}^N\gamma_{-n}^2 E_{n}^k(-ix)E_{n}^k(iy)
+\sum_{n=1}^N\dfrac{k^2}{(n+k)^2}\gamma_{-n}^2 E_{n}^k(ix)E_{n}^k(-iy) 
\\&&\;\;+\sum_{n=1}^N\dfrac{k}{n+k}\gamma_{-n}^2 \Big(E_{n}^k(-ix)E_{n}^k(-iy)
+E_{n}^k(ix)E_{n}^k(iy)\Big).
\end{eqnarray*}
 From  (\ref{69}), the normalization constants satisfy
$$\gamma_{n}^2=\dfrac{n(n+2k)}{(n+k)^2}\gamma_{-n}^2,$$
from which 
$$\gamma_{n}^2+\dfrac{k^2}{(n+k)^2}\gamma_{-n}^2=\gamma_{-n}^2.$$
Consequently, combining the above identities  yields
\begin{eqnarray}\nonumber
 K_N(x,y)&=&\gamma_0^2+\sum_{n=1}^N \gamma_{n}^2  E_n^k(ix)E_n^k(-iy)+\sum_{n=1}^N \gamma_{-n}^2  E_{-n}^k(ix)E_{-n}^k(-iy)
\\&=&\gamma_0^2 +2\sum_{n=1}^N \gamma_{-n}^2  Re\Big(E_n^k(ix)E_n^k(-iy)\Big)+2\sum_{n=1}^N\dfrac{k}{n+k}\gamma_{-n}^2 Re \Big(E_{n}^k(ix)E_{n}^k(iy)\Big).\nonumber\\&&
\end{eqnarray}
Hence $K_N(x,y)$ is real-valued. Let us now  decompose $K_N(x,y)$ as
 $$K_N(x,y)=\sum_{n=0}^N\gamma_{-n}^2 E_{-n}^k(ix)E_{-n}^k(-iy)+\sum_{n=1}^N \gamma_n^2  E_n^k(ix)E_n^k(-iy)=A_N+B_N.$$
 Using the identities (in view of (\ref{Ga}) and \ref{12}),
 $$\gamma_{-n}=\gamma_{n+1},\qquad E_{-n}^k(iu)=e^{iu}E_{n+1}^k(-iu)$$
 we obtain
 \begin{eqnarray*}
 A_N&=&\sum_{n=0}^N\gamma_{-n}^2 E_{-n}^k(ix)E_{-n}^k(-iy)
 \\&=&e^{i(x-y)}
 \sum_{n=0}^N\gamma_{n+1}^2 E_{n+1}^k(-ix)E_{n+1}^k(iy)
 \\&=&e^{i(x-y)}
 \sum_{n=1}^{N+1}\gamma_{n}^2 E_{n}^k(-ix)E_{n}^k(iy)
 \\&=&e^{i(x-y)}
 \sum_{n=1}^{N}\gamma_{n}^2 E_{n}^k(-ix)E_{n}^k(iy)+ 
e^{i(x-y)}\gamma_{N+1}^2  E_{N+1}^k(-ix)E_{N+1}^k(iy)
\\&=&e^{i(x-y)}\overline{ B_N}+e^{i(x-y)}\gamma_{N+1}^2 E_{N+1}^k(-ix)E_{N+1}^k(iy).
 \end{eqnarray*}
 Since  $A_N+B_N$ is real,  we have 
  \begin{eqnarray*}
  K_N(x,y) &=& B_N+e^{i(x-y)}\overline{ B_N}+e^{i(x-y)}\gamma_{N+1}^2
    E_{N+1}^k(-ix)E_{N+1}^k(iy)
 \\&=&\overline{ B_N}+e^{-i(x-y)}B_N+e^{-i(x-y)} \gamma_{N+1}^2E_{N+1}^k(ix)E_{N+1}^k(-iy)
  \\&=&e^{-i(x-y)}\Big(B_N+e^{i(x-y)}\overline{ B_N}\Big)+e^{-i(x-y)}\gamma_{N+1}^2 E_{N+1}^k(ix)E_{N+1}^k(-iy)
\\&=&e^{-i(x-y)} K_N(x,y)
+e^{-i(x-y)}\gamma_{N+1}^2 E_{N+1}^k(ix)E_{N+1}^k(-iy)
\\&&-\gamma_{N+1}^2E_{N+1}^k(-ix)E_{N+1}^k(iy).
 \end{eqnarray*}
This yields the identity
\begin{eqnarray}\nonumber
   K_N(x,y)&=& \gamma_{N+1}^2\;\frac{e^{-i(x-y)}E_{N+1}^k(ix)E_{N+1}^k(-iy)-E_{N+1}^k(-ix)E_{N+1}^k(iy)}{1-e^{-i(x-y)}}
   \\&=&\gamma_{N+1}^2\;\frac{Im\Big(e^{-i(x-y)/2}E_{N+1}^k(ix)E_{N+1}^k(-iy)\Big)} {\sin\left(\frac{x-y}{2}\right)}.\label{EE}
\end{eqnarray}
We conclude by examining the case  $k=0$.  In this situation, the normalization constants reduce to  $\gamma_{n}^2=2\pi$. 
Consequently, the sum under consideration becomes 
  $$\frac{1}{2\pi}\sum_{n=-N}^N e^{in(x-y)}= \frac{1}{2\pi}e^{iN(x-y)}\sum_{n=0}^{2N} e^{-in(x-y)}=\frac{e^{iN(x-y)}
  -e^{-i(N+1)(x-y)}}{1-e^{-i(x-y)}}$$
 which is the classical formula for a finite geometric sum.
  \end{proof}
According to  (\ref{jacobi}) and  (\ref{EE}) the kernel  $K_N(x,y)$ can be written as
\begin{eqnarray}\label{K}\nonumber
K_N(x,y)&=&2\gamma_{N+1}^2\cot\left(\frac{x-y}{2}\right)
   \Big\{\sin xP_{N}^{k+1}(ix)P_{N+1}^k(iy)-\sin yP_{N}^{k+1}(iy)P_{N+1}^k(ix)\Big\}
\\ && \qquad\qquad-\gamma_{N+1}^2\Big\{P_{N+1}^k(ix)P_{N+1}^k(iy)+4\sin x\sin yP_{N}^{k+1}(ix)P_{N}^{k+1}(iy)\Big\}.\nonumber\\&&
\end{eqnarray}
\begin{lem}
There exists a constant $A>0$ independent of $x\in[-\pi,\pi]$ and $n\in \mathbb{N}$ such that
\begin{equation}\label{eq1}
 \gamma_n|\sin^k x E_n^k(ix)| \leq A. 
\end{equation}
\end{lem}
\begin{proof}
A classical estimate for ultraspherical polynomials, as used in \cite{P}   asserts that
 for all  $x\in[-\pi,\pi]$ and $n\in \mathbb{N}$
 \begin{equation}\label{03}
 \frac{|\sin^k xP_n^k(ix)|}{\| P_n^k\|_{2,k}}=\frac{|\sin^k xC_n^k(\cos x)|}{\| C_n^k\|_{2,k}}\leq A
 \end{equation}
where $A$ is independent of $x$ and $n$. Moreover, since
$$\| P_n^k(ix)\|_{2,k}\leq 1/ \gamma_n$$
and for $n\geq 1$
$$| P_{n-1}^{k+1}(ix)\|_{2,k+1}=\|\sin x P_{n-1}^{k+1}(ix)\|_{2,k}\leq1/ \gamma_n$$
then estimate (\ref{eq1}) follows by using (\ref{jacobi}) and (\ref{03}).
\end{proof}
  \begin{lem}\label{lem11}
 Let $a,b\in(0,1)$. Then  the integral 
$$\int_{-\pi}^{\pi}
   \left|\cot\left(\frac{x-y}{2}\right)\left(\left|\dfrac{\sin x}{\sin y}\right|^{a}-\left|\dfrac{\sin x}{\sin y}\right|^{b}\right)\right|dy$$ 
 is uniformly bounded in $x\in [-\pi,\pi]$.
\end{lem}
 
  \begin{proof}
 First, recall the identity
  $$\cot\left(\frac{x-y}{2}\right)= -
 \dfrac{\sin x+\sin y}{\cos x-\cos y}.$$
By symmetry, we may reduce the integral to $y\in[0,\pi]$ and estimate
\begin{eqnarray*}
&&\int_{-\pi}^{\pi}
   \left|\cot\left(\frac{x-y}{2}\right)\left(\left|\dfrac{\sin x}{\sin y}\right|^{a}-\left|\dfrac{\sin x}{\sin y}\right|^{b}\right)\right|dy
  \\&&\qquad\qquad\qquad\leq
  2\int_{0}^{\pi}
 \left|\left|\dfrac{\sin x}{\sin y}\right|^{a}-\left|\dfrac{\sin x}{\sin y}\right|^{b}\right||\cos x-\cos  y|^{-1}\sin ydy
 \\&&\qquad\qquad\qquad\qquad+ 2|\sin x| \int_{0}^{\pi}
 \left|\left|\dfrac{\sin x}{\sin y}\right|^{a}-\left|\dfrac{\sin x}{\sin y}\right|^{b}\right||\cos x-\cos  y|^{-1}dy.
\end{eqnarray*}

Next, make the substitution  $s=\cos  x$ and $t=\cos y$ the  integrals become 
  
  $$  \int_{-1}^{1}|s-t|^{-1}\left|\left(\frac{1-s^2}{1-t^2}\right)^{a/2}
  \left(\frac{1-s^2}{1-t^2}\right)^{b/2}\right|dt $$
and 
$$  \int_{-1}^{1}|s-t|^{-1}\left|\left(\frac{1-s^2}{1-t^2}\right)^{(a+1)/2}
  \left(\frac{1-s^2}{1-t^2}\right)^{(b+1)/2}\right|dt .$$
These integrals are known to be uniformly bounded for  $a,b\in(0,1)$, (see \cite{PP})
 
  \end{proof}
 
\begin{lem}\label{dua}
Let $f\in L^p([-\pi,\pi],dm_k)$ be nonnegative function . Then the integral
$$T(f)(x)=\int_{-\pi}^{\pi}
   f(y)\left|\cot\left(\frac{x-y}{2}\right)\left(\left(\dfrac{\sin x}{\sin y}\right)^{k(2-p)/p}-1\right)\right|dm_k(y)$$ 
define a function in $L^p([-\pi,\pi],dm_k)$ for 
$$2-\frac{1}{k+1}<p<2+\frac{1}{k}.$$
\end{lem}
 \begin{proof}
We use duality. Let $g\in L^{p'}([-\pi,\pi],dm_k)$  be non-negative,where $1/p+1/p'=1$. By Hölder’s inequality, we have

\begin{eqnarray*}
&&\int_{-\pi}^{\pi} g(x)T(f)(x)dm_k(x)\\&=&\int_{-\pi}^{\pi}\int_{-\pi}^{\pi}
  g(x) f(y)\left|\cot\left(\frac{x-y}{2}\right)\left(\left(\dfrac{\sin x}{\sin y}\right)^{k(2-p)/p}-1\right)\right|dm_k(y)dm_k(x)\\
  &\leq& \left\{\int_{-\pi}^{\pi}g^{p'}(x)\left(\int_{-\pi}^{\pi}
  \left|\cot\left(\frac{x-y}{2}\right)\left(\left(\dfrac{\sin x}{\sin y}\right)^{k(2-p)/p}-1\right)\right|\left|\frac{\sin y}{\sin x}\right|^{-1/p}dm_k(y)\right)dm_k(x)\right\}^{1/p'}
\\&& \times \left\{\int_{-\pi}^{\pi} f^{p}(y)\left(\int_{-\pi}^{\pi}
  \left|\cot\left(\frac{x-y}{2}\right)\left(\left(\dfrac{\sin x}{\sin y}\right)^{k(2-p)/p}-1\right)\right|\left|\frac{\sin y}{\sin x}\right|^{1/p'}dm_k(x)\right)dm_k(y)\right\}^{1/p}.
\end{eqnarray*}
Now Lemma \ref{lem11} implies that 
 $$ \sup_{x\in[-\pi,\pi]}\int_{-\pi}^{\pi}
  \left|\cot\left(\frac{x-y}{2}\right)\left(\left(\dfrac{\sin x}{\sin y}\right)^{k(2-p)/p}-1\right)\right|\left|\frac{\sin y}{\sin x}\right|^{-1/p}dm_k(y)<\infty$$
 and 
$$  \sup_{y\in[-\pi,\pi]}\int_{-\pi}^{\pi}
  \left|\cot\left(\frac{x-y}{2}\right)\left(\left(\dfrac{\sin x}{\sin y}\right)^{k(2-p)/p}-1\right)\right|\left|\frac{\sin y}{\sin x}\right|^{1/p'}dm_k(x)<\infty.$$
   Consequently,
$$\int_{-\pi}^{\pi} g(x)T(f)(x)dm_k(x)\leq A \|g\|_{p',k} \|f\|_{p,k}.$$
Taking the supremum over all such $g$ yields $T(f)\in L^p([-\pi,\pi],dm_k)$.
\end{proof}
\begin{lem}
If $f\in  L^p([-\pi,\pi],dm_k)$, $2-\frac{1}{k+1}<p<2+\frac{1}{k}$ then the sequence $(a_n/\gamma_n)_{n\in \mathbb{Z}}$ defined by (\ref{an}) is bounded 
\end{lem}
\begin{proof}
Applying Hölder’s inequality yields
$$|a_n/\gamma_n|\leq  \left\{\int_{-\pi}^{\pi}|f(y)|^p|\sin x|^{2k}dx\right\}^{1/p} \left\{\int_{-\pi}^{\pi}|\gamma_n E_n^k(ix)|^{p'}  |\sin x|^{2k}dx\right\}^{1/p'}$$
By symmetry, we write
$$\int_{-\pi}^{\pi}|\gamma_n E_n^k(ix)|^{p'}  |\sin x|^{2k}dx=4\int_{0}^{\pi/ 2}|\gamma_n E_n^k(ix)|^{p'}  |\sin x|^{2k}dx $$
Using  (\ref{eq1}) , we obtain
$$\int_{0}^{\pi/ 2}|\gamma_n E_n^k(ix)|^{p'}  |\sin x|^{2k}dx \leq  A\int_{0}^{\pi/ 2} |\sin x|^{2k-kp'}dx.$$
The latter integral converges if and only if
$$2k-kp'>-1$$ 
which is equivalent to 
$$p>2-\frac{1}{k+1},$$
this is exactly in the assumed range of $p$. Therefore, we  have that
$$|a_n/\gamma_n|\leq A \|f\|_{p,k}, \qquad \text{for all }\; n\in \mathbb{Z}$$
and the proof is thus complete.
\end{proof}
\begin{lem}
Let $f\in  L^p([-\pi,\pi],dm_k)$, $2-\frac{1}{k+1}<p<2+\frac{1}{k}.$ Then 
$$\limsup_{N\rightarrow+\infty}\|S_N(f)\|_{p,k}<\infty$$
\end{lem}
\begin{proof}
By formulas (\ref{SK}) and  (\ref{K}) the partial sums
 $S_N(f)$ admit the representation
 \begin{eqnarray*}
 S_N(f)(x)&&=-2\gamma_{N+1}^2P_{N+1}^k(ix)\int_{-\pi}^\pi\cot\left(\frac{x-y}{2}\right)
   \sin yP_{N}^{k+1}(iy)f(y) dm_k(y)\\
   &&+2\gamma_{N+1}^2 \sin xP_{N}^{k+1}(ix)\int_{-\pi}^\pi\cot\left(\frac{x-y}{2}\right) P_{N+1}^k(iy)f(y) dm_k(y)\\&&-\left(\frac{a_{N+1}+\overline{a_{N+1}}}{2}\right)P_{N+1}^k(ix)
+i(a_{N+1}-\overline{a_{N+1}})\sin xP_{N}^{k+1}(ix). 
 \end{eqnarray*}
With the  estimate 
(\ref{eq1}),it follows that for $x\neq -\pi,0,\pi$,
  
  $$|S_N(f)(x)|\leq A\gamma_{N+1}\left|\sin^{-k}(x)\int_{-\pi}^\pi\cot\left(\frac{x-y}{2}\right)
   \sin yP_{N}^{k+1}(iy)f(y) dm_k(y)\right|$$ 
   $$+  A\gamma_{N+1}\left|\sin^{-k}(x)\int_{-\pi}^\pi\cot\left(\frac{x-y}{2}\right)
   P_{N+1}^{k}(iy)f(y) dm_k(y)\right|+ A|\sin^{-k}(x)|$$
  Taking the 
$L^p([-\pi,\pi],dm_k)$-norm, we obtain
$$\|S_N(f)(x)\|_{p,k}\leq I_N+J_N+A$$

where 

 $$I_N=
 A\gamma_{N+1}\left\|\sin^{-k}(x)\int_{-\pi}^\pi\cot\left(\frac{x-y}{2}\right)
   \sin yP_{N}^{k+1}(iy)f(y) dm_k(y)\right\|_{p,k}$$
 and 
 $$J_N= A\gamma_{N+1}\left\|\sin^{-k}(x)\int_{-\pi}^\pi\cot\left(\frac{x-y}{2}\right)
   P_{N+1}^{k}(iy)f(y) dm_k(y)\right\|_{p,k}.$$
 We focus on $I_N$ the estimate for $J_N$	 is analogous. Writing the	
 in terms of the Lebesgue measure $dx$, we obtain
  $$I_N=
 A\gamma_{N+1}\left\|\sin^{\frac{2k}{p}-k}(x)\int_{-\pi}^\pi\cot\left(\frac{x-y}{2}\right)
   \sin y|\sin y|^{2k}P_{N}^{k+1}(iy)f(y) dy\right\|_{p}.$$
 We decompose the integrand as
 \begin{eqnarray*}
 &&\gamma_{N+1}|\sin x|^{\frac{2k}{p}-k}|\sin y|^{2k}
   \sin yP_{N}^{k+1}(iy)f(y) = \gamma_{N+1}|\sin y|^{\frac{2k}{p}+k}
  \sin y P_{N}^{k+1}(iy)f(y)
   \\&& \qquad\qquad +\gamma_{N+1}|\sin y|^{\frac{2k}{p}+k}\sin y
   P_{N}^{k+1}(iy)f(y)\left\{ \left|\dfrac{\sin x}{\sin y}\right|^{\frac{2k}{p}-k}-1\right\}.
 \end{eqnarray*}
 Then 
 $$I_N\leq \left\|\int_{-\pi}^\pi\cot\left(\frac{x-y}{2}\right)
 \Big\{\gamma_{N+1}|\sin y|^{\frac{2k}{p}+k}
  \sin y P_{N}^{k+1}(iy)f(y)dy \Big\}\right \|_{p}$$
 $$+\left\|\int_{-\pi}^\pi\cot\left(\frac{x-y}{2}\right)
 \gamma_{N+1}|\sin y|^{\frac{2k}{p}+k}
  \sin y P_{N}^{k+1}(iy)f(y)\left\{ \left|\dfrac{\sin x}{\sin y}\right|^{k(2-p)/p}-1\right\}\right \|_{p}$$
 We  refer to the paper \cite{R} for the estimate
$$ \left\|\int_{-\pi}^\pi\cot\left(\frac{x-y}{2}\right)
 \Big\{\gamma_{N+1}|\sin y|^{\frac{2k}{p}+k}
  \sin y P_{N}^{k+1}(iy)f(y)dy \Big\}\right \|_{p}$$$$\leq A \left\| \gamma_{N+1}
 |\sin y|^{\frac{2k}{p}+k}
  \sin y P_{N}^{k+1}(iy)f(y)\right \|_{p}\leq A \|f\|_{p,k}. $$
  The last inequality follows from 
(\ref{eq1}).
In addition Lemma \ref{dua} with (\ref{eq1}) yield 
 $$\left\|\int_{-\pi}^\pi\cot\left(\frac{x-y}{2}\right)
 \gamma_{N+1}|\sin y|^{\frac{2k}{p}+k}
  \sin y P_{N}^{k+1}(iy)f(y)\left\{ \left|\dfrac{\sin x}{\sin y}\right|^{k(2-p)/p}-1\right\}\right\|_{p}$$
 $$\leq A \left\|\int_{-\pi}^\pi\left|\cot\left(\frac{x-y}{2}\right)
 |\sin y|^{\frac{2k}{p}}
 | f(y)|\left( \left|\dfrac{\sin x}{\sin y}\right|^{k(2-p)/p}-1\right)\right|\right\|_{p}<\infty.$$
 This completes the proof.
\end{proof}
 We now arrive at the main result. As pointed out at the beginning, by combining the preceding lemmas and closely following the argument of Theorem 5.1 in \cite[p. 361]{P}, we establish Theorem \ref{MR}.
\par \par Finally, we provide a counterexample for the case $$p\notin \left(2-\dfrac{1}{k+1}, 2+\frac{1}{k}\right)).$$ 
We consider the same counterexample as  in  \cite[ §9]{P} which demonstrates the failure of the 
 $L^p$ -boundedness outside the above range of $p$.
 Let $$f(x)=(1-\cos x)^{-(k+1)/2}$$
It is easy to see that 
  $f\in L^p([-\pi,\pi],dm_k)$. It therefore suffices to show that
$$\| a_n E_n^k(ix)+a_{-n}E_{-n}(ix)\|_{p,k}$$  
does not tend to $0$ as  $n\rightarrow\infty$. In fact,
 for $n\geq 0$ we have 
$$a_n=\gamma_n^2 \int_{-\pi}^\pi f(x)E_n^k(-ix) |\sin x|^{2k}dx=
2\gamma_n^2\int_{0}^\pi f(x)P_n^k(ix) |\sin x|^{2k}dx, $$
 and by (\ref{cases})
$$a_{-n}=\gamma_{-n}^2 \int_{-\pi}^\pi f(x)E_{-n}^k(-ix) |\sin x|^{2k}dx=
2\gamma_{n+1}^2\dfrac{n+2k}{n+k}\int_{0}^\pi f(x)P_n^k(ix) |\sin x|^{2k}dx, $$
$$=\dfrac{n+k}{n} a_n$$
Now using (\ref{45}) 
$$\| a_n E_n^k(ix)+a_{-n}E_{-n}(ix)\|_{p,k}=2a_n\dfrac{n+k}{n}\|P_n^k(ix)\|_{p,k}\geq A$$
for some constant $A>0$ independent of $n$.
This follows by expressing $P_n^k$ in terms of ultraspherical polynomials via 
  (\ref{ultra})  and using \cite[(9.1)]{P}. The counterexample is therefore proved.

 \end{document}